\documentclass[10pt]{article}
\usepackage{amsmath}
\usepackage{amsfonts}
\usepackage{amssymb}
\usepackage{amsmath}
\usepackage{amsthm}
\usepackage{latexsym}
\usepackage{graphicx}
\usepackage{epstopdf}
\usepackage{tikz}
\usetikzlibrary{matrix,arrows,decorations.pathmorphing}

\newtheorem{theorem}{Theorem}[section]

\newtheorem{definition}[theorem]{Definition}

\newtheorem{lemma}[theorem]{Lemma}

\newtheorem{proposition}[theorem]{Proposition}
\newtheorem{remark}[theorem]{Remark}

\begin{document}


\title{Homology computation for a class of contact structures on $T^3$}

\author{Ali Maalaoui$^{(1)}$ \& Vittorio Martino$^{(2)}$}
\addtocounter{footnote}{1}
\footnotetext{Department of Mathematics, Rutgers University - Hill Center for the Mathematical Sciences 110 Frelinghuysen Rd., Piscataway 08854-8019 NJ, USA. E-mail address: {\tt{maalaoui@math.rutgers.edu}}}
\addtocounter{footnote}{1}
\footnotetext{SISSA, International School for Advanced Studies,
via Bonomea, 265 - 34136 Trieste, Italy. E-mail address:
{\tt{vmartino@sissa.it}}}
\date{}
\maketitle

\vspace{5mm}

{\noindent\bf Abstract} {\small We consider a family of tight contact forms on the three-dimensional torus and we compute the relative Contact Homology by using the variational theory of critical points at infinity. We will also show local stability.}

\vspace{5mm}

\noindent

\section{Introduction}
In this paper we will consider a family of contact structures on the torus $T^3$ and we will compute their relative Contact Homology. We will set the problem in a suitable variational framework and we will use the techniques developed by A.Bahri in his works \cite{bahri1}, \cite{bahri2}, \cite{bahri3} and with Y.Xu in \cite{bahri-xu 2007}.\\
Let us then define the torus $T^3=S^1 \times S^1 \times S^1$, parameterized with coordinates
$$x,y,z \in (0,2\pi) \times (0,2\pi) \times (0,2\pi)$$
and by identifying $0$ and $2\pi$. On the torus we consider the family of infinitely many differential one-forms defined by
$$\alpha_n=\cos(n z) dx + \sin(n z)dy, \qquad n\in \mathbb{N}$$
A direct computation shows that
$$d\alpha_n=n\sin(n z) dx \wedge dz - n\cos(n z) dy \wedge dz$$
and consequently
$$\alpha_n \wedge d\alpha_n=-n dx \wedge dy \wedge dz$$
Therefore, for every $n\in \mathbb{N}$, $(T^3, \alpha_n)$ is a contact manifold, with contact structure given by $\sigma_n=\ker(\alpha_n)$. In particular by a classification result due to Y.Kanda \cite{kanda 1997}, we have that every tight contact structure on $T^3$ is contactomorphic to one of the $\alpha_n$; moreover for $n\neq m$, the contact structures $\sigma_n$ and $\sigma_m$ are not contactomorphic.\\
Our main result is the following:
\begin{theorem}\label{mainteo}
Let $g$ be an homotopy class of the two-dimensional torus $T^{2}$, then for every $n\in \mathbb{N}$, we have
\begin{equation}
H_{k}(\alpha_{n},g)=\left \{
\begin{array}{llcc}
\mathbb{Z}\oplus\ldots \oplus \mathbb{Z} \text{ n times, } & \text{if } k =0,1  \\
0 , & \text{if } k >1
\end{array}
\right.
\end{equation}
\end{theorem}

\noindent
We will prove that the homology is locally stable, namely we will consider small perturbations of the forms in the family $\{\alpha_n\}$ and we will show the our computations still hold.\\

\noindent
We will also show some additional algebraic relations between the contact homologies of the family $\{\alpha_n\}$: in particular we will exhibit an equivariant homology reduction under the action of $\mathbb{Z}_k$, that is for every integer $k$, we will prove the existence of a morphism
$$f_{*}:H_{*}(\alpha_{k n},g)\longrightarrow H_{*}(\alpha_{n},g)$$
that corresponds to an equivariant homology reduction under the action of the group $\mathbb{Z}_{k}$, namely $$H_{*}(\alpha_{n},g)=H_{*}^{\mathbb{Z}_{k}}(\alpha_{k n},g)$$
Finally, in the last section, we will consider the case of a more general $2$-torus bundles over $S^1$
$$T^{2}\times \mathbb{R}/(x,y,z)=(A(x,y),z+2\pi)$$
where $A$ is a given matrix in $SL_{2}(\mathbb{Z})$. We will consider the families of contact forms introduced by Giroux \cite{GR} of the following form
$$\alpha_{h}=\cos(h(z))dx+\sin(h(z))dy$$
with $h$ a strictly increasing function. We will prove that for the related contact structures Theorem \ref{mainteo} still holds.\\

\noindent
Other results on Homology computations are in the works of F.Bourgeois \cite{bour1} and F.Bourgeois-V.Colin \cite{bour2}, where the authors compute the homology using the cylindrical contact homology which coincides with our result if we disregard the degeneracy. Also in his thesis dissertation E.Lebow \cite{lebow} computed the embedded contact homology for $2$-torus bundles which appears to be very different from the result that we find here.

\section{General setting of the problem}

\noindent
Here we briefly introduce the general framework developed by A.Bahri. Let $(M,\alpha)$ be a three-dimensional, compact and orientable manifold without boundary. In order to apply the theory we will need to assume that there exists a suitable non singular vector field in $\ker(\alpha)$ that will allow us to complete a sort of Legendre duality, namely we assume that:
$$
\begin{array}{ll}
(i) & \exists \; v\in TM, \mbox{a non-vanishing vector field, such that} \; v\in ker(\alpha);\\
\\
(ii) & \mbox{the non-singular dual differential form} \;\beta(\cdot):=d\alpha(v,\cdot) \mbox{ is a contact }\\
    & \mbox{form on $M$ with the same orientation than $\alpha$}.\\
\end{array}
$$
We will show that hypotheses $(i)$ and $(ii)$ hold in our framework. We explicitly note that for general contact structures it is not known if the previous assumptions are fulfilled: for instance in \cite{io} the existence of such a $v$ satisfying $(i), (ii)$ is established for the first contact form of the family of overtwisted contact form on the three-dimensional sphere $S^3$ defined by Gonzalo-Varela in \cite{GonVar}; in particular in \cite{io} the (explicit) existence of such a $v$ satisfying $(i)$ is proven for all the overtwisted forms defined by Gonzalo-Varela on $S^3$, but with this $v$ hypothesis $(ii)$ holds only for the first contact form of this family: another $v$ might work for the other forms.\\

\noindent
Next we define the action functional
\begin{equation}\label{action}
J(x) =\int_0^1 \alpha(\dot{x})
\end{equation}
on the subspace of the $H^1$-loops on $M$:
$$ C_\beta = \{x \in H^1(S^1;M)\; s.t. \; \beta(\dot{x})=0; \; \alpha(\dot{x}) = \mbox{strictly positive constant}\}$$
Now if $\xi\in TM$ denotes the Reeb vector field of $\alpha$, i.e.
\begin{equation}\label{reeb}
\alpha(\xi)=1, \qquad d\alpha(\xi,\cdot)=0
\end{equation}
then the following result by A.Bahri-D.Bennequin holds \cite{bahri1}:
\begin{theorem}\label{Jtheorem}
$J$ is a $C^2$ functional on $C_\beta$ whose critical points are of finite Morse index and are periodic orbits of $\xi$.
\end{theorem}

\noindent
Now, for the sake of computations, we rescale $v$ such that
$$\alpha \wedge d\alpha=\beta \wedge d\beta$$
then in particular we have:
$$d\alpha(v,[\xi,v])=-1$$
Moreover we introduce the functions $\tau$ and $\bar\mu$ defined by:
$$[\xi,[\xi,v]]=-\tau v$$
and
$$\bar\mu=d\alpha(v,[v,[\xi,v]])$$
so that the Reeb vector field of $\beta$ is
$$w=\bar\mu \xi - [\xi,v]$$
We note that a general tangent vector $z$ to $M$ reads as
$$z=\lambda \xi + \mu v + \eta w$$
for some functions $\lambda, \mu, \eta$. Also, a curve $x$ belongs to $C_\beta$ if
$$\dot{x}=a \xi + b v$$
for some function $b$ and with $a$ being a positive constant. Therefore, if $z$ is tangent to $C_\beta$ at $x$, it holds:
$$
\left\{
\begin{array}{c}
\dot{\overline{\lambda + \overline{\mu} \eta}} = b \eta-\int_0^1 b \eta \\
\\
\;\quad \dot{\eta} = a \mu - b \lambda\\
\\
\lambda, \mu, \eta \quad  \mbox{1-periodic}
\end{array}
\right.
$$
The second derivative of $J$ at a critical point $x$ ($b=0$) reads as:
\begin{equation}\label{secondderivative}
J''(x)\cdot z \cdot z= \displaystyle \int_0^1 \dot{\eta}^2-a^2 \eta^2 \tau
\end{equation}
We will also need the transport maps $\psi_s$ and $\phi_s$ of $\xi$ and $v$ respectively, namely the
one parameter group of diffeomorphism generated by the flows
\begin{equation}\label{transport maps psi}
\left\{
\begin{array}{l}
\displaystyle \frac{d}{ds}\big(\psi_s(x)\big)=\xi_{\psi_s(x)} \\
\\
\psi_0(x)=x
\end{array}
\right.
\end{equation}
and
\begin{equation}\label{transport maps phi}
 \left\{
\begin{array}{l}
\displaystyle \frac{d}{ds}\big(\phi_s(x)\big)=v_{\phi_s(x)} \\
\\
\phi_0(x)=x
\end{array}
\right.
\end{equation}

\noindent
The major difficulties that show up in the variational analysis of this functional are the lack of compactness (that is the Palais-Smale condition does not hold) and the loss of the Fredholm condition. In fact the linearized operator is not Fredholm in general and this is a serious issue in the the Morse theoretical methods since one cannot apply the implicit function theorem anymore and therefore the Morse lemma does not hold. We know that the Fredholm assumption is violated for the standard contact structure $\alpha_{0}$ on $S^{3}$ and the first exotic structure of Gonzalo and Varela \cite{GonVar}. There is a simple criteria to check if violation occurs or not based on some properties of the transport map $\phi$ of the special legendrian vector field $v$. First, by looking at the functional $J$ in the larger space $\mathcal{C}_{\beta}^{+}=\{x\in \mathcal{L}_{\beta}| \alpha(\dot{x})\geq 0\}$, we notice that it remains insensitive to the introduction of a ``back and forth'' $v$ piece. If we consider a modified functional in the following way
$$\tilde{J}(x)=\int_{0}^{1} \alpha(\dot{x})(t)dt+\delta \log(1+\int_{0}^{1}|b(t)|dt)$$
then, as it is shown in \cite{bahri6}, it is Fredholm since one has a control on $b$ in this case. Now, let us take a curve that is transverse to $v$, and at a point $x(t_{0})$ we introduce an a back and forth $v$ piece of length $s$ and let us call $x_{\epsilon}$ the curve obtained by introducing a small ``opening'' piece of length $\epsilon$ between the two $v$ pieces. Then we have
$$J(x_{\epsilon})=J(x)-\epsilon(\alpha_{x(t_{0})}(d\phi_{-s}(\xi))-1)+o(\epsilon).$$

\noindent
Thus if there exists $s>0$ such that $\alpha (\phi_{-s}(\xi))>1$, then we would have a decreasing direction from the level $J(x)$ and we would be able to bypass a critical point without changing the topology even though it has a finite Morse index, and this is due to the loss of the Fredholm condition.
Now we can state the following:
\begin{lemma}[Bahri \cite{bahri7}]\label{Fredholm}
If $\phi^{*}_{-s}(\alpha)(\xi) < 1$, for every $s \neq 0$, then $J$ satisfies the Fredholm condition.
\end{lemma}

\noindent
We will show that in our framework Fredholm does not hold. In fact, we will see that we will have situations for which there will exist $s\neq 0$, such that $\phi^{*}_{-s}(\alpha)(\xi)= 1$.\\

\begin{center}
\includegraphics[scale=.6]{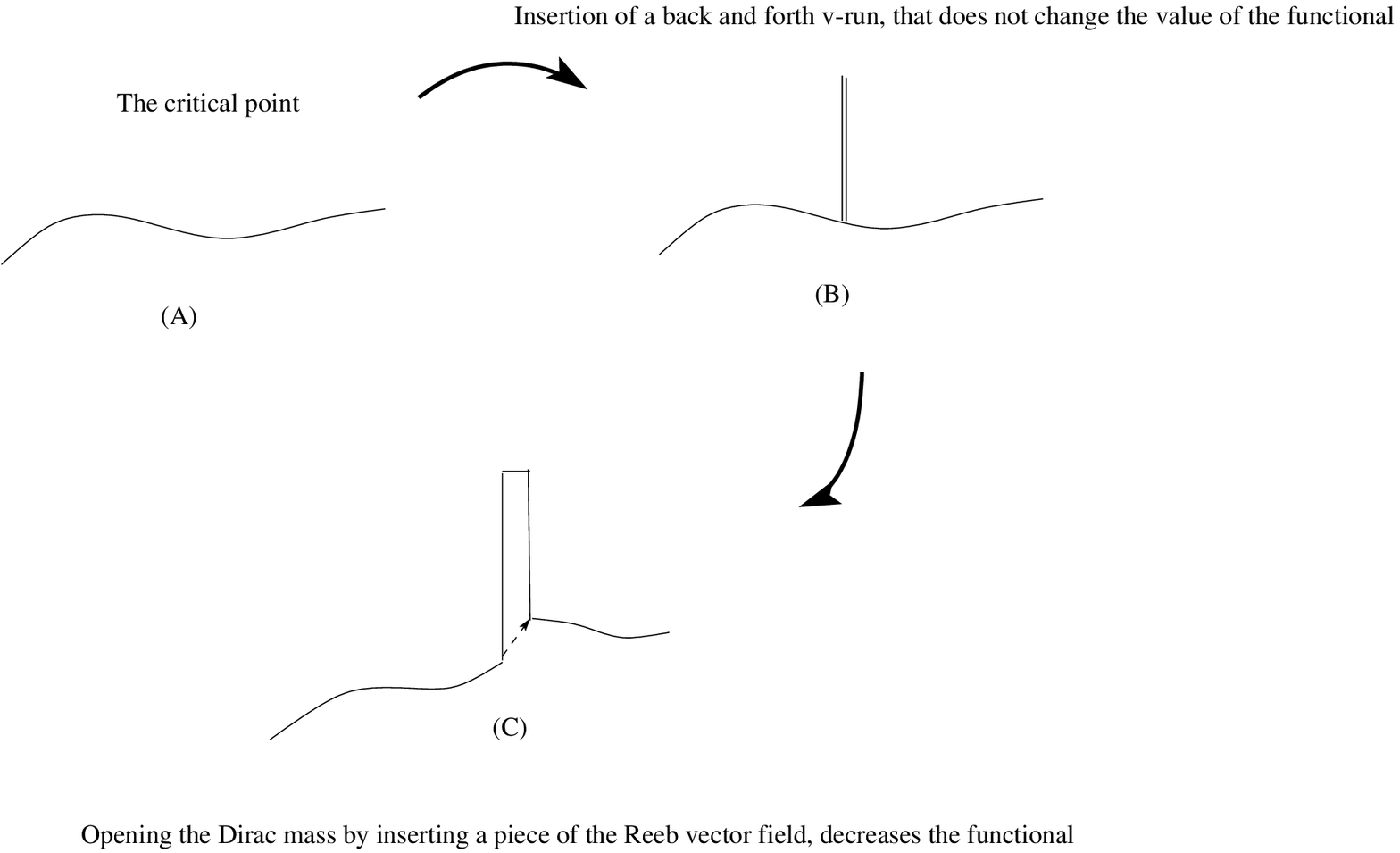}
\end{center}

\noindent
In order to prove Theorem (\ref{mainteo}), we will first compute explicitly all the quantities defined in this variational framework for our family of contact forms $\{\alpha_n\}$.\\
Later, since for our model we will show that the second derivative of $J$ will have a direction of degeneracy corresponding to the action of $[\xi,v]$, the critical points will come in circles. This degeneracy will be removed by a small perturbation of the functional in a neighborhood of the critical points in order to ``break the symmetry''.\\
Then, in order to compute explicitly the homology in our framework, we need to worry about the non-compactness due to the presence of asymptotes. To deal with that we will show that the the critical points at infinity have always higher energy so that they cannot interact with our critical points, that is cancelations cannot occur. Hence the problem will come down in counting the number of periodic orbits. The idea is the same as in the theory of critical points at infinity, namely after compactifying the space, by adding the asymptotes, the classical Morse theory tells us that indeed $\partial^{2}=0$, but in this situation the boundary operator $\partial$ has two components $\partial = \partial_{per}+\partial_{\infty}$. The operator $\partial_{per}$ counts the number of pseudo-gradient flow lines between periodic orbits (actual critical points) and $\partial_{\infty}$ counts the flow lines between critical points at infinity and periodic orbits. Therefore to show that we have compactness in our homology theory, we need that $\partial_{per}^{2}=0$. Now if we compute $$\partial^{2}=\partial_{per}^{2}+\partial_{\infty}^{2}+\partial_{per}\partial_{\infty}+\partial_{\infty}\partial_{per}$$
Hence if we show that $\partial_{per}\partial_{\infty}+\partial_{\infty}\partial_{per}=0$ when applied to periodic orbits, then compactness holds.\\
Finally, since the Fredholm condition is violated, we will show however that the homology is locally stable along isotopies.\\

\noindent
In the last two sections we will first show also some additional algebraic relations between the contact homologies of the family $\{\alpha_n\}$ and then we will consider the case of a more general $2$-torus bundles over $S^1$.

\section{Proof of Theorem (\ref{mainteo})}

\noindent
Here we compute explicitly all the quantities defined in the previous section for our family of contact forms $\{\alpha_n\}$.\\
The Reeb vector field $\xi_n$ is given by:
$$\xi_n=\cos(n z) \partial_x + \sin(n z) \partial_y$$
Now if we set:
$$v_n=\frac{1}{n} \partial_z$$
then we have $v_n\in \ker(\alpha_n)$ and
$$\beta_n(\cdot):=d\alpha_n(v_n,\cdot)=-\sin(n z) dx + \cos(n z) dy$$
Since
$$d\beta_n=n\cos(n z) dx \wedge dz + n\sin(n z) dy \wedge dz$$
and
$$\beta_n \wedge d\beta_n=-n dx \wedge dy \wedge dz$$
therefore, with this choice of the vector field $v_n$, we obtain that hypotheses $(i)$ and $(ii)$ are fulfilled; moreover
$$ \alpha_n \wedge d\alpha_n=\beta_n \wedge d\beta_n$$
Furthermore we compute
$$[\xi_n,v_n]=\sin(n z) \partial_x - \cos(n z) \partial_y$$
thus $[\xi_n,[\xi_n,v_n]]=0$ and so $\tau_n$ identically vanishes. Also, since
$$w_n=-[\xi_n,v_n]$$
is the Reeb vector field for $\beta_n$, then $\bar\mu$ must be zero.\\
Therefore, by using (\ref{secondderivative}), the second derivative of $J$ at a critical point $x$ reduces to:
\begin{equation}\label{secondderivativetorus}
J''(x)\cdot z \cdot z= \displaystyle \int_0^1 \dot{\eta}^2
\end{equation}
Notice that since $\tau=0$ we have a direction of degeneracy corresponding to $\eta$ constant. So the critical points will come in circles generated by the action of $[\xi,v]$. The next Lemma shows how to perturb the functional near the critical sets, in order to ``break the symmetry'' and avoid degeneracy. First let us compute explicitly also the transport maps (\ref{transport maps psi}) and (\ref{transport maps phi}):
\begin{equation}\label{transport maps psi torus}
\psi_s(x,y,z)=\big(\cos(n z)s+x, \sin(n z)s+y, z \big)
\end{equation}
and
\begin{equation}\label{transport maps phi torus}
\phi_s(x,y,z)=\big(x, y, z+\frac{s}{n} \big)
\end{equation}

\begin{lemma}
There exists a perturbed functional $J_\varepsilon$, for small $\varepsilon >0$, in a neighborhood of the critical sets of $J$, such that $J_\varepsilon$ is equal to $J$ outside this neighborhood, and it has exactly 2 critical points inside it: a minimum and a maximum.
\end{lemma}

\begin{proof}
From the equation (\ref{transport maps psi torus}) we see that we have periodicity for the orbits of $\xi$ if there exists $z$ such that $\tan(nz)$ is rational. With the same $z$ also the orbits of $[\xi,v]$ are closed and since $[\xi,v]$ is transported along $\xi$ we have that the set of critical points has two different $S^1$-actions: the first is the natural one due to the translation on time along the curve itself, and the second one due to the action of $[\xi,v]$ that gives rise to the degeneracy.\\
Now we want to describe the tangent space of $C_\beta$ at a critical point. We know that if
$$Z=\lambda \xi + \mu v + \eta [\xi,v]$$
is tangent to $C_\beta$, we need only the function $\eta$ to describe completely the tangent space; in particular at a critical point $x=a\xi$, it holds $\dot{\eta} = a \mu$.\\
In addition, the set of critical points is a submanifold of $C_\beta$, endowed with the $S^1$-action given by $[\xi,v]$: if $Z$ is tangent to this submanifold, we have $\mu=0$ and therefore $\eta$ is constant. Moreover, the normal space to the submanifold is given by the functions $\eta \in H^1$ that are orthogonal to the constants, namely the normal space to the submanifold is generated by the vector fields $Z$ (tangent to $C_\beta$) having
$$\eta \in H^1(S^1;\mathbb{R}), \qquad s.t. \qquad \int_0^1\eta(t)dt=0 $$
Since the second derivative of $J$ at a critical point $x$ reads as $J''(x)\cdot Z \cdot Z= \int_0^1 \dot{\eta}^2$, we see that for a non vanishing normal variation, we have $J''(x)\cdot Z \cdot Z>0$ and this shows indeed that the critical sets are isolated.\\
Hence we can split the tangent space to $C_\beta$ at a critical point $x$ in the following way:
$$T_x C_\beta=\{\theta\} \oplus \{\eta\}, \qquad \theta \in \mathbb{R},\quad  \eta \in H^1(S^1;\mathbb{R}), \quad \int_0^1\eta=0 $$
Now we want to construct a tubular neighborhood around the orbit of $[\xi,v]$: so by means of the exponential map ($C_\beta$ is an Hilbert manifold) we will consider the neighborhood around the critical set given by
$$\theta + s \eta, \quad s\in [0,1],\quad \theta \in \mathbb{R}, \quad \eta \in H^1, \quad \int_0^1\eta=0, \quad \|\eta\|_{H^1}\leq \delta $$
Therefore our functional reads in this neighborhood as $\tilde J(\theta,\eta)$, and we note that by construction $\frac{\partial \tilde J}{\partial \theta} \equiv 0$. Now we will perturb it in the following way
$$\tilde J_\varepsilon(\theta,\eta)=\tilde J(\theta,\eta) + \varepsilon w\big(\|\eta\|_{H^1}\big)f(\theta)$$
where $f$ is a smooth function on $S^1$ having exactly 2 critical points, and $w(r)$ is a cut-off function that vanishes outside $|r|\geq \delta$ and it is equal to 1 for $|\delta|\leq \delta/2$. Now by choosing suitable small constants $\varepsilon$, $\delta$ and the bump function $w$, we get that the functional $\tilde J_\varepsilon$ is equal to the old functional outside this neighborhood, and it has exactly 2 critical points inside it: a minimum and a maximum.
\end{proof}

\noindent
Now we recall that in this setting A.Bahri introduced different pseudo-gradient flows. For instance in \cite{bahri6}, \cite{bahri7} it is was shown that the natural $L^{2}$-pseudo-gradient for $J$ on $C_\beta$ is not the right flow to consider since at the blow-up time there is the presence of an absolutely continuous part adding up to the Diracs therefore another flow was constructed that does the right decreasing. We will consider the second flow defined in \cite{bahri2}.\\
It is shown for this flow the existence of critical points at infinity made by alternating $v$- and $\xi$-pieces. We define the set
$$\Gamma_{2k}=\left\{ \gamma \in C_{\beta} , ab=0 \right \}$$
that is the set of curves in $C_{\beta}$ made by $k$ $v_n$-pieces and $k$ $\xi_n$-pieces. Then we consider the set of variation at infinity, namely
$$\bigcup_{k \geq 0} \Gamma_{2k}$$
and on this set we define the functional at infinity
$$J_{\infty} (\gamma)=\sum_{k=0} ^{k=\infty} a_{k}$$
The critical points of this functional are what we call critical points at infinity, and we have and the exact characterization for them. First we need the following two definitions:
\begin{definition}
A $v$-jump between two points $x_0$ and $x_1 = x(s_1)$, $s_1\neq0$, is a $v$-jump between
conjugate points if it holds:
$$\big( \phi_{s_1}^\ast \alpha \big)_{x_1} = \alpha_{x_0} $$
In other words conjugate points are points on the same $v$-orbit such that the form $\alpha$ is transported onto itself by the transport map along $v$.
\end{definition}

\begin{definition}
A $\xi$-piece $[x_0; x_1]$ of orbit is characteristic if $v$ completes exactly a number $k \in \mathbb{Z}$ of half revolutions from $x_0$ to $x_1$.
\end{definition}

\noindent
It holds (see \cite{bahri2}):
\begin{proposition}\label{classification}
A curve in $\bigcup_{k \geq 0} \Gamma_{2k}$ is a critical point at infinity if it satisfies one of the following assertions:
$$
\begin{array}{ll}
(1) & \mbox{the $v$-jumps are between conjugate points. These critical points are denoted in the sequel  ``true''} \\
    & \mbox{critical points at infinity;}\\
\\
(2) & \mbox{the $\xi$-pieces have characteristic length, and in addition } \\
    & \mbox{the $v$-jumps send $ker(\alpha)$ to itself.}\\
\end{array}
$$
\end{proposition}

\noindent
In our case we see from the transport equation (\ref{transport maps psi torus}) along $\xi_n$, that we cannot have $\xi$-pieces with characteristic length, thus all the critical points at infinity are ``true''. Also, we see from the transport equation (\ref{transport maps phi torus}) along $v_n$, that each point has $n$ conjugate points corresponding to the translation along $z$ by $\frac{2\pi}{n}$. Next we check the validity of the Fredholm condition. We have:
\begin{lemma}
The Fredholm assumption is violated.
\end{lemma}
\begin{proof}
By using Lemma (\ref{Fredholm}) and by a straightforward computation, if we just compute the transport of $\xi_n$ along $v_n$, we get
$$\big(\phi_{s}^{\ast} \alpha_n \big) (\xi_n) = \cos(s) \leq 1$$
Since we can have $s\neq0$ such that equality occurs, then Fredholm does not hold.
\end{proof}

\noindent
Next, in order to compute explicitly the Homology it suffices to show that there is no interaction between periodic orbits and critical points at infinity, in the sense that there are no flow lines among them. This is what we prove in the following:
\begin{lemma}
There is no interaction between the periodic orbits of $\xi_n$ and the critical points at infinity.
\end{lemma}
\begin{proof}
First from the classification result Lemma (\ref{classification}), the critical points at infinity for our model are just periodic orbits of $\xi_n$ with some additional back and forth $v$-jumps of length multiple of $\frac{2\pi}{n}$. An interesting case happens when $n=1$ since each point in the orbit of $\xi_n$ can be conjugate only to itself along $v_n$, so we have periodic orbits linked by $v$-cycles.\\
Now by knowing the trivial splitting of the fundamental group of the torus, that is
$$\Pi_{1}(T^{3})=\mathbb{Z} \oplus \mathbb{Z} \oplus \mathbb{Z}$$
we denote by
$$P_{3}:\Pi_{1}(T^{3}) \longrightarrow \mathbb{Z}$$
the natural projection on the third component, namely: if
$$[\gamma] \in \Pi_{1}(T^{3}), \qquad [\gamma] =(m,n,k)$$
then $P_{3}([\gamma])=k$.
Next we explicitly note now that if $x$ is a periodic orbit of the Reeb vector field $\xi_n$, then $P_{3}([x])=0$; moreover any $v$-cycle will add a pure third component, thus it will have projection non zero. In particular if $x_{\infty}$ is a critical point at infinity, let us suppose with $m$ $v$-cycles (with orientation) attached to a periodic orbit of $\xi$, then $$P_{3}([x_{\infty}])=\sum_{i=1}^{m} k_{i}$$
where $k_{i}$ is the number of iterations of the $k$-th $v$-cycle counted with its orientation.
Therefore we deduce that a periodic orbit and a critical point at infinity can interact if and only if $P_{3}([x_{\infty}])=0$.
Notice that since the strict index of the periodic orbits is zero, and the index of the critical points at infinity is at least 1, then trivially we have that $\partial_{per}^{2}=0$. But notice that we have a richer structure here built by the tower of critical points at infinity above each critical point.
\end{proof}

\noindent
Now we will prove the main Theorem.
\begin{proof}(of Theorem \ref{mainteo})\\
Let $g$ be an homotopy class in $T^2$, then $g$ reads as:
$$g=m \overline{x}+l \overline{y}$$
where $\overline{x}$ and $\overline{y}$ are the generators of $\Pi_{1}(T^2)$. Hence since
$$\xi_n=\cos(n z) \partial_x + \sin(n z) \partial_y$$
we get that $g$ contains exactly $n$ periodic orbits of $\xi_n$ (in fact we have $n$ circles of critical points). By breaking the symmetry each circle can be seen as a min and a max with zero boundary operator between them.
Therefore
\begin{equation}
H_{k}(\alpha_{n},g)=\left \{
\begin{array}{llcc}
\mathbb{Z}\oplus\ldots \oplus \mathbb{Z} \text{ n times, } & \text{if } k =0,1  \\
0 , & \text{if } k >1
\end{array}
\right.
\end{equation}

\end{proof}

\noindent
In the last part of this section we will show that our computations are ``locally stable'', that is we are interested in small perturbations of the contact forms in the family $\{\alpha_n\}$. Thus, let us suppose that $\alpha$ is a form in the previous family, and let us consider a perturbed form $\tilde \alpha:= u\alpha$, where $u\in C^2(M,\mathbb{R})$ and $\|1-u\|_{C^2}$ is small. Hence we will get a new functional $\tilde J$ whose critical points $\tilde x$ will be in a $L^\infty$ neighborhood of the original ones $x$. We will show that in fact $J(\tilde x)\geq J(x)$. In order to do that we first use a result in \cite{bahri6}: every curve $x_0 \in C_\beta$ in a $L^\infty$ neighborhood of a critical point $x$ can be represented by a curve $x_1$ made only by pieces of $\xi$-orbit and finitely many $\pm v$-jumps, and in addition this is a minimizing process, i.e. $J(x_1)\leq J(x_0)$. In particular in our situation, in order to stay in a given homotopy class, the $v$-jumps need to be small, moreover since the $\xi$-pieces have the $z$ component constant, we need the sum of the $\pm v$-jumps to be zero: hence we can think to have finitely many nearly ``Dirac masses'' placed on the original critical point $x$. Now we can obtain this broken curve $x_1$ from a critical point $x$ by pushing along a deformation vector $Z$ having $\eta$ such that:
$$\ddot\eta=\sum_{i=1}^k \pm A_i(\delta_{t_i^-}-\delta_{t_i^+})$$
where $k$ is the number of the ``Dirac masses'', $A_i>0$ represent the jump in the $v$ direction, and $t_i^-,t_i^+$ are the times where the jumps occur. If we compute the second variation along this $Z$, we get:
$$J''(x)ZZ=\int_0^1 \dot\eta^2 = -\int_0^1 \eta \ddot\eta = \sum_{i=1}^k \pm A_i\big(\eta(t_i^+)-\eta(t_i^-)\big)$$
Now let us consider disjoint intervals $[T_i^-,T_i^+]$, each of them containing $[t_i^-,t_i^+]$, with $T_{i-1}^+=T_i^-$ and $T_i^+=T_{i+1}^-$. By a direct computation we find
$$J''(x)ZZ=\sum_{i=1}^k A_i^2 \big(t_i^+ - t_i^-\big)\Big(1-\frac{t_i^+ - t_i^-}{T_i^+ - T_i^-} \Big) >0$$
Therefore $Z$ is a strictly increasing direction for $J$ and this proves the local stability.

\noindent
Finally, we want to show a strict relation between our structures and some spaces of configurations. So, given a periodic orbit, let us consider the set $\tilde\Gamma_{2k}$ made by the periodic orbit with attached $+k$ $v$-orbits and $-k$ $v$-orbits. Studying this space corresponds to understand the configuration space of signed particles on $S^{1}$. This was studied in a paper by D. McDuff \cite{Mc} in which she gives a full description of the space of configuration of signed particles, denoted by $C^{\pm}$, as follows:

\begin{theorem}[McDuff \cite{Mc}]
If $M$ is a manifold without boundary, there is a homotopy equivalence between $C^{\pm}(M)$ and $\Gamma^{\pm}$ the space of compactly supported sections from $M$ to $E^{\pm}$ the bundle over $M$ constructed by taking at each point of $x\in M$ the set $S_{x}\times S_{x}/D$. Here $S_{x}$ is the unit sphere in the tangent space at $x$ and $D$ the diagonal.
\end{theorem}

\noindent
For instance, one sees that the space $\tilde\Gamma_{4}$ (made by the periodic orbit with two $v$ periodic orbits attached to it with opposite orientations)  has the topology of $S^{2}$ with two points identified. The identification comes from the fact that if the two $v$-orbits coincide at the same point, they cancel each other. In particular, in the case $n=1$, the space $\tilde\Gamma_{4}$ coincides with the space $\Gamma_{4}$ with the $v$-pieces having opposite orientations.\\

\begin{center}
\includegraphics[scale=.3]{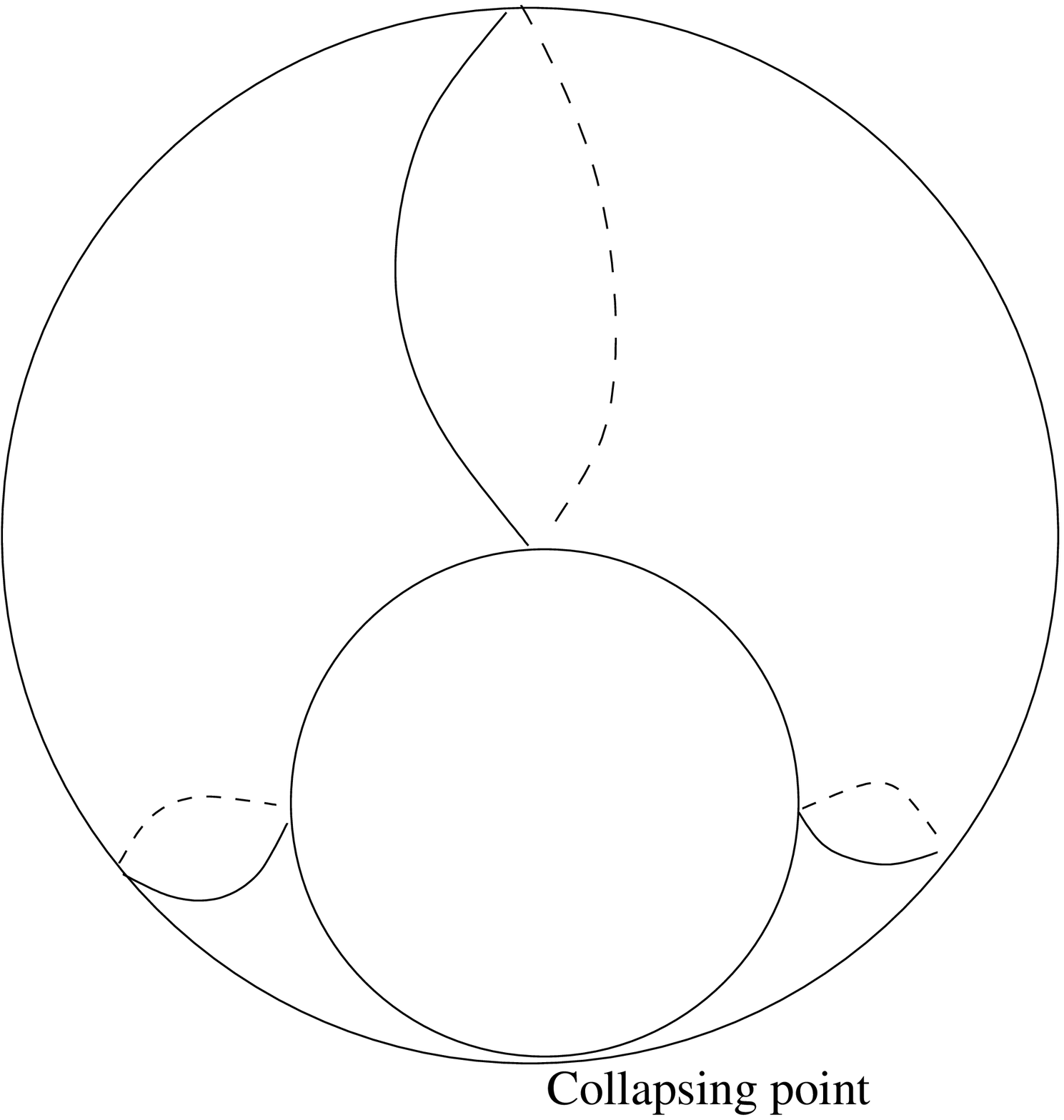}
\end{center}

\noindent
Indeed because of the extra $S^{1}$-action that we have, the full structure can be seen as in the figure below.\\

\begin{center}
\includegraphics[scale=.4]{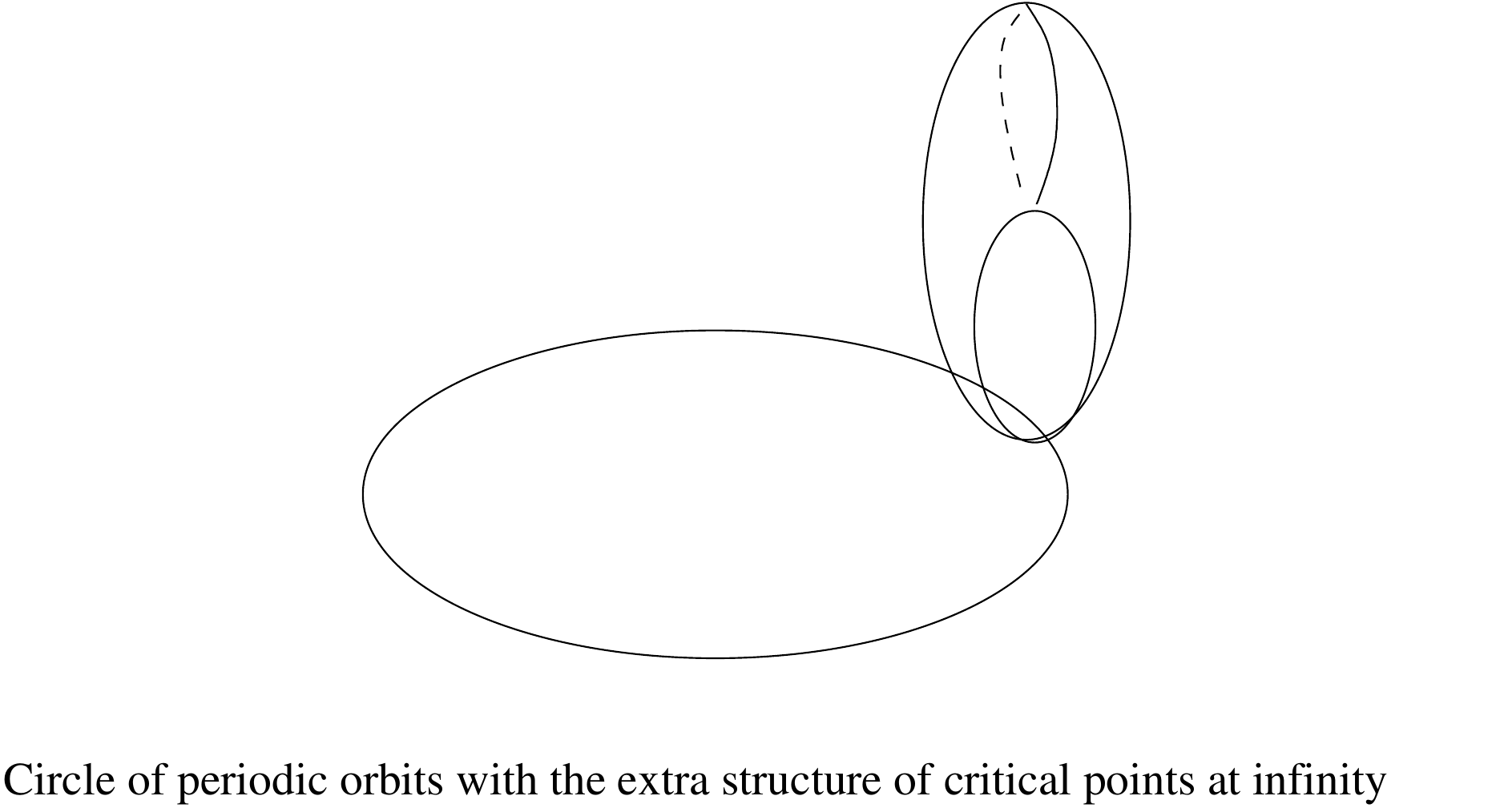}
\end{center}

\section{More Structures}

\noindent
In this section we will give some algebraic relations between the different contact homologies of the family $\{\alpha_{n}\}$.

\begin{theorem}
Let $p$ and $k$ be a positive integers, then there exists a morphism
$$f_{*}:H_{*}(\alpha_{kp},g)\longrightarrow H_{*}(\alpha_{p},g).$$
Moreover, this homomorphism corresponds to an equivariant homology reduction under the action of the group $\mathbb{Z}_{k}$, that is $$H_{*}(\alpha_{p},g)=H_{*}^{\mathbb{Z}_{k}}(\alpha_{kp},g).$$
\end{theorem}
\begin{proof}
Let us consider the action of the group $\mathbb{Z}_{k}$ on the torus by translating the third component, namely the action generated by $f(x,y,z)=(x,y,z+\frac{2\pi}{k})$. We notice that the contact form is invariant under $f$, that is
$$f^{*}\alpha_{kp}=\cos(kp (z+\frac{2\pi}{k}))dx+\sin(kp (z+\frac{2\pi}{k}))dy=\alpha_{kp}$$
Therefore also the functional $J_{\alpha_{kp}}$ is invariant under this action. We recall that at the chain level the boundary operator $\partial$ counts the number of orbits of a decreasing pseudo-gradient for $J$. For two periodic orbits $x_{1}$ and $x_{2}$ of $\xi$ we define $\langle x_{1},x_{2}\rangle$ as the number of gradient flow lines from $x_{1}$ to $x_{2}$, if the index difference is one. With this notation we have that
$$\partial x_{1} =\sum_{i_{x_{k}}=i_{x_{1}}-1} \langle x_{1},x_{k}\rangle x_{k}$$
Next we define
$$C_{n}(\alpha_{p},g):=Crit_{n}(J_{\alpha_{p}},g) \otimes \mathbb{Z}$$
where $Crit_{n}(J_{\alpha_{p}},g)$ is the set of critical points of $J_{\alpha_{p}}$ in the homotopy class $g\in \Pi_{1}(T^{2})$ with Morse index $n$. We notice that $$Crit_{n}(J_{\alpha_{kp}},g)/\mathbb{Z}_{k}=Crit_{n}(J_{\alpha_{p}},g)$$
Therefore
$$C_{n}^{\mathbb{Z}_{k}}(\alpha_{kp},g):=Crit_{n}(J_{\alpha_{kp}},g)/\mathbb{Z}_{k} \otimes \mathbb{Z}=C_{n}(\alpha_{p},g)$$
so we can define the surjective group homomorphism
$$f_{*}:C_{*}(\alpha_{kp},g)\longrightarrow C_{*}(\alpha_{p},g)$$
induced on the quotient by the group action of $\mathbb{Z}_{k}$ on the generators. We claim that this is indeed a chain map.
In fact the boundary operator on the quotient chain is defined by $$\partial_{\mathbb{Z}_{k}}\tilde{x}_{1}=\sum_{\tilde{x}_{i}\in Crit_{n-1}(J_{\alpha_{kp}},g)/\mathbb{Z}_{k}} \sum_{j=1}^{k} <x_{1},x_{i}^{j}> \tilde{x}_{i}$$
where $\tilde{x}_{1}=f_{n}(x_{1})$ and  $\{x_{i}^{j}\}_{j}=f^{-1}_{*}(\tilde{x}_{i})$. It is easy to see now that $\partial_{\mathbb{Z}_{k}}^{2}=0$ and $f_{*}$ is indeed a chain map by construction. In fact, we have $$f_{*}\partial x_{1}=f_{*}(\sum_{x_{i}\in Crit_{n-1}(J_{\alpha_{kp}},g)} <x_{1},x_{i}>x_{i})$$
$$=\sum_{x_{i}\in Crit_{n-1}(J_{\alpha_{kp}},g)} <x_{1},x_{i}>f_{*}(x_{i}),$$ by grouping the terms with the same image under $f$ we get that $$\partial_{\mathbb{Z}_{k}}\tilde{x}_{1}=f_{*}\partial.$$
Now using this fact we have $$\partial_{\mathbb{Z}_{k}}^{2}\tilde{x}_{1}=\partial_{\mathbb{Z}_{k}}f_{*}(\partial x_{1})=f_{*}\partial^{2}x_{1}=0.$$
Thus it descends to a morphism in the homology level.
\end{proof}

\noindent
Then one has the following commuting diagram :

\begin{tikzpicture}
[back line/.style={densely dotted}, cross line/.style={preaction={draw=white, -, line width=6pt}}]
\matrix (m) [matrix of math nodes,row sep=3em, column sep=3em,text height=3.5ex,text depth=3.25ex]
{
& H_{*}(\alpha_{pq},g) & & H_{*-1}(\alpha_{pq},g) \\
H_{*}(\alpha_{p},g) & & H_{*-1}(\alpha_{p},g) \\
& H_{*}(\alpha_{q},g) & & H_{*-1}(\alpha_{q},g) \\
H_{*}(\alpha_{1},g) & & H_{*-1}(\alpha_{1},g) \\
};
\path[->]
(m-1-2) edge node[auto]{$\partial_{pq}$} (m-1-4) edge node[auto]{$f_{*}^{q}$}
(m-2-1) edge [back line] node[auto]{$f_{*}^{p}$} (m-3-2)
(m-1-4) edge node[auto]{$f_{*-1}^{p}$} (m-3-4)
edge node[auto]{$f_{*-1}^{q}$}(m-2-3)
(m-2-1) edge [cross line]  node[auto]{$\partial_{p}$} (m-2-3)
edge node[auto]{$f_{*}^{p}$}(m-4-1)
(m-3-2) edge [back line]  node[auto]{$\partial_{q}$}(m-3-4)
edge [back line]  node[auto]{$f^{*}_{q}$}(m-4-1)
(m-4-1) edge node[auto]{$\partial$} (m-4-3)
(m-3-4) edge node[auto]{$f_{*-1}^{q}$}(m-4-3)
(m-2-3) edge [cross line] node[auto]{$f_{*-1}^{p}$}  (m-4-3);
\end{tikzpicture}

$$$$
\noindent
Moreover if we consider one of the faces of the previous diagram we have for $p_{1},\cdots, p_{k}$, $k$ positive integers:
$$$$

\begin{tikzpicture}
\matrix(m)[matrix of math nodes,
row sep=2.6em, column sep=2.8em,
text height=1.5ex, text depth=0.25ex]
{\cdots&H_{*}(\alpha_{p_{k}\cdots p_{1}},g)&H_{*-1}(\alpha_{p_{k}\cdots p_{1}},g)&\cdots\\
&\vdots&\vdots&\\
\cdots&H_{*}(\alpha_{p_{1}},g)&H_{*-1}(\alpha_{p_{1}},g)&\cdots\\
\cdots&H_{*}(\alpha_{1},g)&H_{*-1}(\alpha_{1},g)&\cdots\\};
\path[->]
(m-1-1) edge (m-1-2)
(m-1-2)edge node[auto] {$\partial_{p_{k}\cdots p_{1}}$} (m-1-3)
(m-1-3) edge (m-1-4)
(m-1-2)edge (m-1-3)
(m-1-3) edge (m-1-4)
(m-1-2)edge node[auto]{$f_{*}^{p_{k}}$} (m-2-2)
(m-1-3) edge node[auto] {$f_{*-1}^{(p_{k}}$} (m-2-3)
(m-3-1) edge (m-3-2)
(m-3-2) edge node[auto] {$\partial_{p_{1}}$} (m-3-3)
(m-3-3) edge (m-3-4)
(m-2-2)edge node[auto]{$f_{*}^{p_{2}}$} (m-3-2)
(m-2-3)edge node[auto]{$f_{*-1}^{p_{2}}$} (m-3-3)

(m-4-1) edge (m-4-2)
(m-4-2) edge node[auto] {$\partial$} (m-4-3)
(m-4-3) edge (m-4-4)
(m-3-2)edge node[auto]{$f_{*}^{p_{1}}$} (m-4-2)
(m-3-3)edge node[auto]{$f_{*-1}^{p_{1}}$} (m-4-3);
\end{tikzpicture}

\section{Torus Bundles}

\noindent
We consider now the case of more general $2$-torus bundles over $S^{1}$. Given a matrix $A\in SL_{2}(\mathbb{Z})$, we define the space $$Y_{A}=T^{2}\times \mathbb{R}/(x,y,z)=(A(x,y),z+2\pi).$$
We recall that the fundamental group of $Y_{A}$, is $\Pi_{1}(Y_{A})=\mathbb{Z}\times \mathbb{Z}\rtimes_{A} \mathbb{Z}$. From the work of Giroux \cite{GR}, we know that these spaces contains infinitely many contact structures, given by a fixed contact form $\alpha$.
The construction of such structures starts by taking a strictly increasing function $h$ and considering the contact form $\alpha_{h}$ on $\mathbb{R}^{3}$ defined by
$$\alpha_{h}=\cos(h(z))dx+\sin(h(z))dy$$
We state then the result of Giroux  as follow:
\begin{theorem}[\cite{GR}]
Let $A$ be a matrix in $SL_{2}(\mathbb{Z})$ then:
$$
\begin{array}{ll}
a) & \mbox{For every $n\geq 0$ there exists a contact structure on $\mathbb{R}^{3}$ given by the 1-form}
\end{array}
$$
$$
  \alpha_{h_n}=\cos(h_n(z))dx+\sin(h_n(z))dy,
$$
$$
\begin{array}{ll}
   & \mbox{that is invariant under the action of the fundamental group of $Y_{A}$ and}\\
   & \mbox{the increasing function $h$ satisfies}:
\end{array}
$$
$$
   2\pi n \leq h_n(z+2\pi)-h_n(z) < 2\pi (n+1)
$$
$$
\begin{array}{ll}
b) & \mbox{The contact structure descends to a contact structure on $Y_{A}$, depending}\\\
   & \mbox{only on $n$  up to isotopy}\\
\\
c)& \mbox{All these contact structures are homotopic as plane fields on $Y_{A}$.}\\
\end{array}
$$
\end{theorem}

\begin{remark}
We explicitly note that the family of contact forms we considered in the first part of the paper correspond to the choice of $h_n(z)=nz$, with $A=I_2$.
\end{remark}
\noindent
We are going to compute for these contact forms all the quantities needed in order to apply the variational method.
First we have that the Reeb vector field of $\alpha_{h_n}$ is given by
$$\xi_{h_n}=\cos(h_n(z)) \partial_x + \sin(h_n(z)) \partial_y$$
Then, by straightforward computations we get the following
\begin{lemma}
The 1-form $\beta_{h_n}=d\alpha_{h_n}(v_{n},\cdot)$ is a contact form with the same orientation than $\alpha_{h_n}$ on $Y_{A}$ with
$$v_{h_n}=\frac{1}{h_{n}'(z)}\partial_{z}$$
Therefore hypotheses $(i)$ and $(ii)$ are fulfilled and
$$ \alpha_{h_n} \wedge d\alpha_{h_n}=\beta_{h_n} \wedge d\beta_{h_n}$$
Moreover $\tau_{h_n}$ and $\bar\mu_{h_n}$ are zero.
\end{lemma}

\noindent
Also
\begin{lemma}
The transport maps $\psi_s$ and $\phi_s$ of $\xi_{h_n}$ and $v_{h_n}$ respectively are given by:
\begin{equation}\label{transport maps psin torus}
\psi_s(x,y,z)=\Big(\frac{\cos(h_n(z))s}{h'_{n}(z)}+x, \frac{\sin(h_n(z))s}{h'_{n}(z)}+y, z \Big)
\end{equation}
and
\begin{equation}\label{transport maps phin torus}
\phi_s(x,y,z)=\Big(x, y, z+\frac{s}{h'_{n}(z)} \Big)
\end{equation}
\end{lemma}

\noindent
Now we can check the Fredholm condition. We have:
\begin{lemma}
The Fredholm assumption is violated for all the contact forms $\alpha_{h_n}$.
\end{lemma}
\begin{proof}
By using again Lemma (\ref{Fredholm}), if we compute the transport of $\xi_{h_n}$ along $v_{h_n}$, we get
$$\big(\phi_{s}^{\ast} \alpha_{h_n} \big) (\xi_{h_n}) = \cos\Big(h_n\big(z+\frac{s}{h'_{n}(z)}\big)-h_n(z)\Big) \leq 1$$
Hence Fredholm does not hold.
\end{proof}

\noindent
Moreover by the transport equations for $\xi_{h_n}$ and $v_{h_n}$ we see that there are no $\xi$-pieces with characteristic length. Regarding the conjugate points, different scenarios might happen. We will distinguish two cases.\\

\noindent
\emph{Case} 1: the conjugate points are in different fibers.\\
In fact for two points to be conjugate we need to have
$$h_{n}(z+\frac{s}{h'_{s}(z)})-h_{n}(z)=0 mod (2\pi)$$
By using the fact that
$$2n\pi<h_{n}(z+2\pi)-h_{n}(z)\leq 2(n+1)\pi$$
there exist $n$ values $s$, with $0<s\leq 2\pi h'_{n}(z)$, such that $h_{n}(z+\frac{s}{h'_{n}(z)})-h_{n}(z)$ is a multiple of $2\pi$, and this corresponds to conjugate points in different fibers (see fig 1).\\

\begin{center}
\includegraphics[scale=.4]{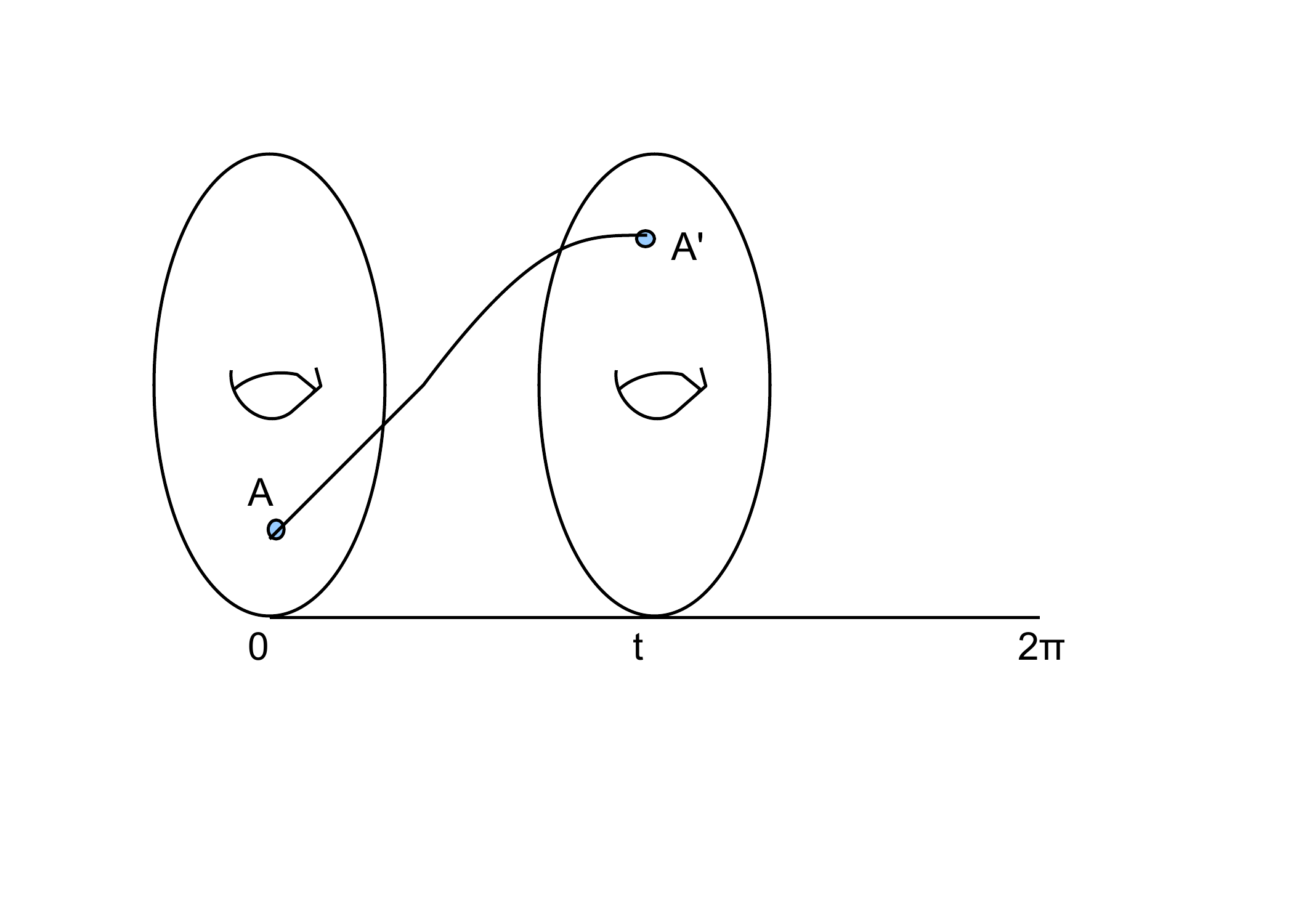}
\end{center}

\noindent
\emph{Case} 2 : the conjugate points are in the same fiber.\\
This case happens in the particular situation when
$$h_{n}(z+2\pi)-h_{n}(z)=2(n+1)\pi$$
and so the conjugate point in the same fiber is achieved when $s=2\pi h'_{n}(z)$ (see fig 2).\\

\begin{center}
\includegraphics[scale=.4]{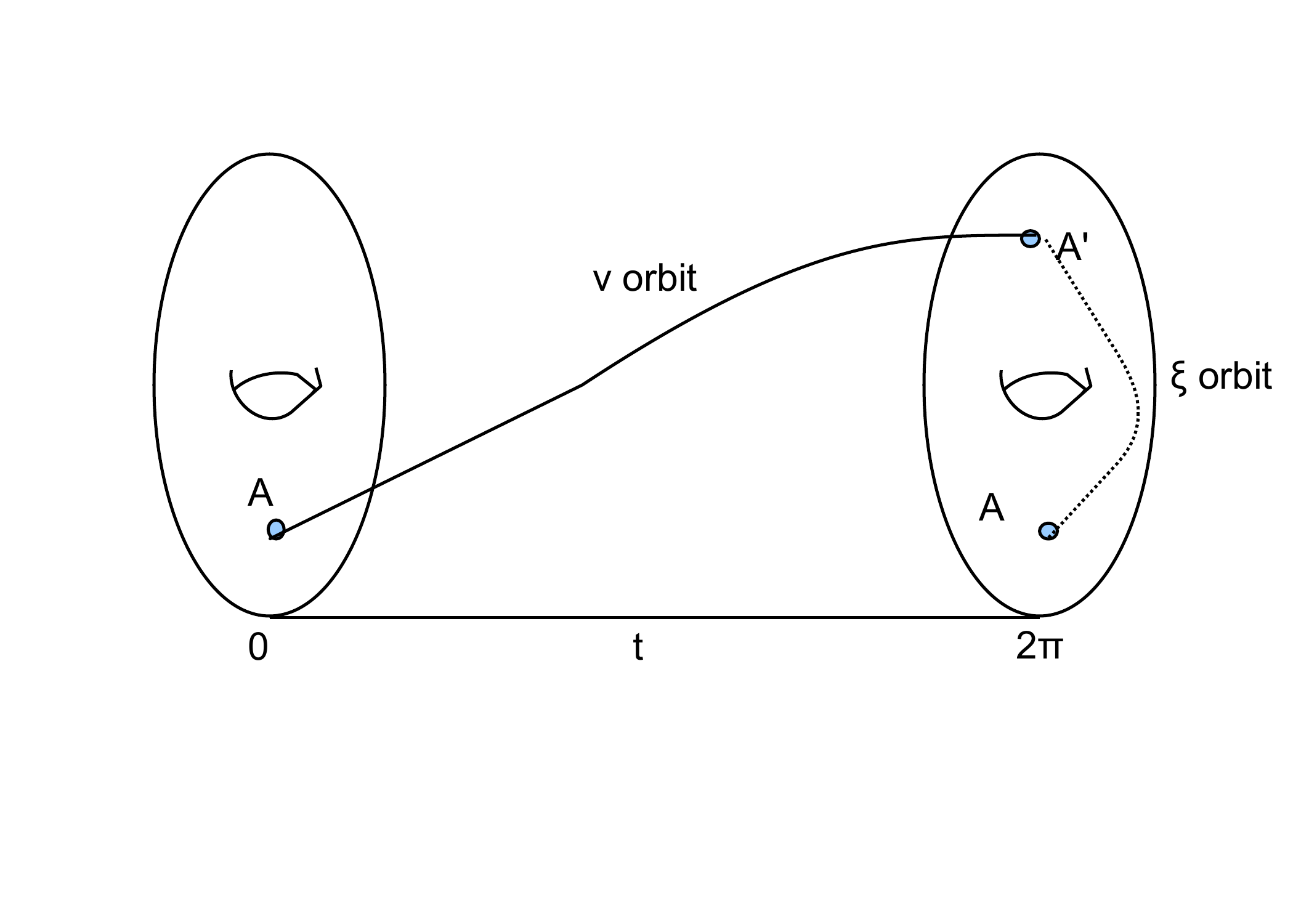}
\end{center}

\noindent
Another important thing to notice is that the orbits of the Reeb vector field are tangent to the fibers, thus if two conjugate points are in a different fibers we need more that one $v$-piece to be able to close the curve as critical point at infinity.\\
On the other hand, in the case where the conjugate points are in the same fiber, we can close the orbit by a $\xi$-piece, but in this situation we are in a different homotopy class (as in the case of $T^{3}$ considered in the previous sections). Hence in order to be able to close the curve staying in the homotopy class containing the periodic orbits, we need to have at least one $v$-piece in the opposite direction.\\
Then with the same reasoning as in the case of the torus $T^3$, we have that the index of the critical points at infinity is strictly greater than zero in a given homotopy class.
In order to compute the homology, we need just to find the index of the periodic orbits, but since $\tau$ is zero, we can proceed as in the previous case of the torus $T^3$: therefore there is no interaction between the periodic orbits and the critical points at infinity.
Now let us fix an homotopy class $g\in \Pi_{1}(T^{2})$, with $g=(a,b)\in \mathbb{Z}\oplus \mathbb{Z}$. Since the periodic orbits are all tangent to the fibers then the periodicity condition is equivalent to $\tan(h_{n}(z))=\frac{a}{b}$ and this corresponds to $n$ periodic orbits.
Finally we have proved the following
\begin{theorem}
For any given contact structure of the form $\alpha_{h_{n}}$ on $Y_{A}$, if $g\in \Pi_{1}(T^{2})$, we have:
\begin{equation}
H_{k}(\alpha_{h_{n}},g)=\left \{
\begin{array}{llcc}
\mathbb{Z}\oplus\ldots \oplus \mathbb{Z} \text{ n times, } & \text{if } k =0,1  \\
0 , & \text{if } k >1
\end{array}
\right.
\end{equation}

\end{theorem}

\end{document}